\documentclass[11pt,oneside]{amsart}
\usepackage{graphicx} % Required for inserting images
\usepackage{amsmath}
\usepackage{amssymb}
\usepackage{eucal}
\usepackage{amsthm}
\usepackage{amscd}
\usepackage{hyperref}
\usepackage{soul}
\usepackage{url}
\usepackage{skull}
\usepackage{mathrsfs}
 \usepackage[authormarkup=none]{changes}

\usepackage{tikz}
\usepackage{tikz-cd}
\usetikzlibrary{shapes.geometric, arrows}
\usetikzlibrary{shapes,arrows}
\usepackage[latin1]{inputenc}
\usetikzlibrary{shapes,arrows}
\usetikzlibrary{shapes.multipart}
\usepackage{verbatim}
\usepackage{amssymb}
\usepackage[]{amsmath, amsthm, amsfonts,graphicx, amscd}
\usepackage[all,cmtip]{xy}
\input amssym.def \input amssym
\usepackage{forest}
\usepackage{mdframed}
\usepackage{framed}

\newcommand{\acl}{\operatorname{acl}}
\newcommand{\dcl}{\operatorname{dcl}}

\newtheorem{thm}{Theorem}[section]

\newtheorem{prop}[thm]{Proposition}
\newtheorem{cor}[thm]{Corollary}

\newtheorem {lem}[thm]{Lemma}

\newtheorem{conj}[thm]{Conjecture}

\theoremstyle{remark}
\newtheorem{rem}[thm]{Remark}

\newtheorem{np*}{Non-Proof}

\theoremstyle{definition}
\newtheorem{defn}[thm]{Definition}

\theoremstyle{plain}

\newcommand{\gen}[1]{\left\langle#1\right\rangle}

\def\Ind{\setbox0=\hbox{$x$}\kern\wd0\hbox to 0pt{\hss$\mid$\hss} \lower.9\ht0\hbox to 0pt{\hss$\smile$\hss}\kern\wd0}
\def\Notind{\setbox0=\hbox{$x$}\kern\wd0\hbox to 0pt{\mathchardef \nn=12854\hss$\nn$\kern1.4\wd0\hss}\hbox to 0pt{\hss$\mid$\hss}\lower.9\ht0 \hbox to 0pt{\hss$\smile$\hss}\kern\wd0}
\def\ind{\mathop{\mathpalette\Ind{}}}

\numberwithin{equation}{section}

\newcommand{\m}{\mathbb }
\newcommand{\mc}{\mathcal }
\newcommand{\mf}{\mathfrak }

\newcommand{\tp}{\operatorname{tp}}

\newcommand{\diff}{\operatorname{diff}}
 \makeatletter
\let\@wraptoccontribs\wraptoccontribs
\makeatother

\begin{document}

\title{On the number of independent solutions of \\ algebraic differential equations}
\author{James Freitag}
\address{James Freitag \\ University of Illinois Chicago, Department of Mathematics, Statistics, and Computer Science, 851 S. Morgan Street, Chicago, IL, USA 60607-7045 }
\email{jfreitag@uic.edu}

\author{Omar Le\'on S\'anchez}
\address{Omar Le\'on S\'anchez \\ Department of Mathematics, University of Manchester, Oxford Road, Manchester, M13 9PL, United Kingdom}
\email{omar.sanchez@manchester.ac.uk}

\author{Wei Li}
\address{Wei Li \\ State Key Laboratory of Mathematical Sciences, Academy of Mathematics and Systems Science, Chinese Academy of Sciences,  Beijing 100190, China}
\email{liwei@mmrc.iss.ac.cn}

\author{Joel Nagloo}
\address{Joel Nagloo \\ University of Illinois Chicago, Department of Mathematics, Statistics, and Computer Science, 851 S. Morgan Street, Chicago, IL, USA 60607-7045}
\email{jnagloo@uic.edu}
\thanks{JF is supported by NSF grant DMS-2452197 and CAREER award 1945251. OLS is partially supported by EPSRC grant EP/V03619X/1. WL is partially supported by  the Strategic Priority Research Program of Chinese Academy of Sciences under Grant XDA0480503, National Key R\&D Program of China 2023YFA1009401, NSFC Grants (No.12122118, 12288201), CAS Project for Young Scientists in Basic Research with Grant No. YSBR-008. JN is partially supported by NSF grant DMS-2348885. }

\contrib[with an appendix by]{Michael F. Singer}

\date{\today} 

\begin{abstract}
We prove a conjecture of Kumbhakar, Roy, and Srinivasan (2024) on the classification of order one differential equations, and a conjecture of Kumbhakar and Srinivasan (2025) on higher order equations. Both conjectures involve bounds for the number of independent solutions of the equation and are shown to be results of recent work in differential Galois theory using model theoretic techniques. In both cases, stronger versions of the conjectures hold when working over the field of constants (i.e., when the equation is autonomous). We then use inverse Galois theory to show that the bounds in the conjectures are optimal when working over a differential field which is differentially finitely generated over its constant subfield. We also show how recent results of Jaoui and Moosa (2024) on abelian reductions of differential equations can be used to recover some of the work of Kumbhakar and Srinivasan (2025). 
\end{abstract}

\maketitle

\tableofcontents

\section{Introduction}
In \cite{freitag2022any} it is shown that given $X$ an order one differential equation over a differential field $(k,\delta)$ of characteristic zero, if \emph{any four} distinct nonalgebraic solutions %$y_1, y_2, y_3, y_4$ of $X$ 
are algebraically independent over $k$, then \emph{all} distinct nonalgebraic solutions are algebraically independent over $k$. When $k$ equals its constant subfield $C$, four can be replaced by two. 

In the context of considering solutions in an arbitrary differential field extending $k$ (or even arbitrary meromorphic solutions), there is no meaningful way to change the above \emph{universal} quantification over solutions to an existential one - one might always build a differential field extension of $k$ containing arbitrarily many independent solutions of $X$. However, when restrictions are placed upon the field extensions of $k$ under consideration, such an existential version becomes possible. In \cite[Conjecture 1.2]{kumbhakar2024classification}, the authors consider only differential field extensions $K$ of $k$ adding no new constants, conjecturing: 

\begin{conj} \label{ordoneconj} \cite{kumbhakar2024classification}
A first order differential equation over $k$ (respectively, an autonomous differential equation over $C$) is not of general type if and only if it has at most three (respectively, at most one) algebraically independent solutions in any given differential field extension of $k$ having $C$ as its field of constants.
\end{conj}

In the above conjecture, one can assume without loss of generality that the differential field $k$ is algebraically closed and that the differential field extending $k$ having $C$ as its field of constants is $k^{\diff}$, the differential closure of $k$. %{\color{blue} (OLS: Are we now assuming that $C$ is algebraically closed? I am happy to assume this, as it does simplify otherwise annoying (little) arguments. Or even better outright assume that $k$ is algebraically closed?)}

We will see that the conjecture is closely related to an existential version of the above result of \cite{freitag2022any}: Let $X$ be a first order differential equation over $k$. If \emph{there exist} four algebraically independent solutions in $k^{\diff}$ to $X$, then \emph{there exist} infinitely many algebraically independent solutions of $X$ in $k^{\diff}$. When $X$ is autonomous, four can be replaced by two. Modulo understanding the model theoretic content of the definition of an order one equation of general type, the above conjecture of Kumbhakar, Roy, and Srinivasan is essentially equivalent to this existential version of the above result of \cite{freitag2022any}. 

We prove Conjecture \ref{ordoneconj} and some extensions in Section \ref{ordonesection}. In Section \ref{orderoneoptimal}, we show that over any nonconstant differential field which is (differentially) finitely generated over its constants, the condition of having four algebraically independent solutions can not be improved to any smaller number of independent solutions (this goes via the inverse problem in differential Galois theory). In a recent article \cite{kumbhakar2025strongly}, Kumbhakar and Srinivasan independently give a proof of the conjecture. One advantage of the proof given here is that we obtain more information in the case that the equation is of general type. Our proof also applies in the case of partial differential equations whose solutions lie in a differential field with finitely many commuting derivations. 

In \cite{kumbhakar2025strongly}, Kumbhakar and Srinivasan give a higher order analog of Conjecture \ref{ordoneconj}, proving the result for equations of order \emph{at most four.} We next introduce the notation to understand their higher order analog of Conjecture \ref{ordoneconj}. 

Consider the equation
\begin{equation} \label{arbintro} f(x, x' , \ldots x^{(n)}) = 0 
\end{equation} where $f$ is an irreducible polynomial over $k$. Let $y$ be a nonalgebraic solution to equation~\eqref{arbintro}. Let $\mf S_{k,f,y}$ be the set of solutions of (\ref{arbintro}) satisfying: 
\begin{enumerate} 
\item For any $y_1 \in \mf S_{k,f,y}$, there is a differential $k$-isomorphism between the differential fields $k \langle y \rangle$ and $k \langle y_1 \rangle $ that maps $y$ to $y_1$. 
\item For any distinct $y_1, \ldots , y_m \in \mf S_{k,f,y}$, the field of constants of $k \langle y_1 , \ldots , y_m \rangle$ is $C$ and $\text{tr.deg}_k k \langle y_1 , \ldots , y_m \rangle = \sum _{i=1} ^m \text{tr.deg}_k k \langle y_i \rangle $.
%$C=\mc C(k)$. 
%\item For any $y_1 , \ldots , y_2 , \ldots , y_m \in \mf S_{k,f,y}$, $\trdeg _k (k \langle y_1 , \ldots , y_m \rangle) = \sum _{i=1} ^m \trdeg _k (k \langle y_i \rangle ).$
\end{enumerate} 

\begin{rem} 
The above definition is given in \cite{kumbhakar2025strongly}, and there may be some questions around the specific language defining the set  $\mf S_{k,f,y}$, since the implied uniqueness of this set can only be up to isomorphism of differential fields over $k$. We should have perhaps defined $\mf S_{k,f,y}$ simply as a maximal set of independent generic solutions of equation~\eqref{arbintro} for which the field of constants of the differential field over $k$ generated by the solutions is equal to the field of constants of $k$. Concretely, the issue is that if one identifies specific functions in some $F \supset k$, there is not a \emph{unique} maximal set of functions. For instance it would be problematic if $y_1$ were to be a transcendental constant multiple of $y_2$, even though both $y_1$ and $y_2$ themselves could be members of such sets (though not together in the same set). In any case, it is easy to see that the set is well-defined up to $k$-differential isomorphism. 
%{\color{blue}(OLS: could we just say that we take a maximal set of independent generic solutions in the differential closure $\bar k$? Otherwise, we could say a maximal set of independent generic solutions in the monster model (recalling that you only call an element a solution if it adds no new constants.)}
\end{rem}

The result of \cite{kumbhakar2025strongly} is:

\begin{thm} \label{generalizethis}
    Let $k$ be an algebraically closed differential field and $y$ a nonalgebraic solution of \eqref{arbintro}. 
    \begin{enumerate}
        \item $\mf S_{k,f,y}$ is a finite set if and only if there is $k$-irreducible differential subfield $F$ of $k \langle y \rangle$ that can be embedded in a strongly normal extension of $k$ and that either $F= k\langle w \rangle $, where $w$ is a nonalgebraic solution of some Riccati differential equation over $k$, or $F$ is an abelian extension of $k$ in which case $\mf S_{k,f,y}=\{ y \}.$
        \item Suppose that $k= C$. 
        Then, $\mf S_{k,f,y}=\{y\}$ if and only if there is an intermediate differential field $C \subset M \subset C\langle y \rangle $ 
        such that either $M=C(x)$, where 
        $x'=1$ or
        $x'=cx$ 
        for some nonzero $c \in C$, or $M$ is an abelian extension of $C$.
        \item Suppose that $n \leq 4$. Then, $\mf S_{k,f,y}$ is a finite set if and only if it has at most $n+2$ elements. 
    \end{enumerate}
    \end{thm}

In Section \ref{higherorderks}, we provide a proof of part (3) of the above theorem for equations of any order. In fact, after the appropriate translation between the model-theoretic and differential-algebraic languages, this follows from (model-theoretic) results in \cite{freitag2023bounding} which we formulate in line with the more recent \cite{freitag2025finite}. When the field $k$ is a field of constants, using \cite{jaoui2025abelian}, we observe a stronger form with the bound of $1$ in place of $n+2$. The proof again works for systems of partial differential equations in a suitable context. 

In section \ref{optimalks}, we show that over any nonconstant differential field which is (differentially) finitely generated over its constants, one cannot improve the bound of $n+2$ in the above theorem. As in the case of order one, the proof of this optimality is established via the inverse differential Galois theory problem over the field $k$. Though we only need a weak form of the inverse problem, we work over a more general sort of differential field than is usual for the problem over which, for instance, we could not find a suitable reference in the literature. To fill this gap, we have included an appendix by Michael F. Singer solving the inverse problem in this context. 

The results of Kumbhakar and Srinivasan in \cite{kumbhakar2025strongly} are powered by a strong classification result for $k$-irreducible subfields of a strongly normal extension. Some of these results are well-known to model-theorists; in fact, after the appropriate translation, they can be viewed as an instance of recent work of Jaoui and Moosa \cite{jaoui2025abelian}. In Section \ref{abredJM}, we explain this connection, showing how results in \cite{jaoui2025abelian} (restricted to the context of differential fields) yield results of \cite{kumbhakar2025strongly}. We do, however, note that the proofs in \cite{kumbhakar2025strongly} involve interesting and different techniques than the ones in \cite{jaoui2025abelian}.

%In certain ways the results of \cite{kumbhakar2025strongly} are stronger, while in other ways the results of \cite{jaoui2025abelian} are stronger, so the connection strengthens the results of both works. 

\subsection{A note on authorship} 
The authors first considered the conjecture of Kumbhakar, Roy, and Srinivasan at the 2024  Beijing Model Theory Conference, where we outlined the proof contained in Section \ref{ordonesection}. We wrote the details of the proof during a visit of Le\'on S\'anchez to Chicago in May 2025. Subsequently, Kumbhakar and Srinivasan released their own proof of the conjecture in \cite{karhumaki2025primitive} by rather different methods in which they also formulated the variant of the conjecture for higher order equations. 

Kumbhakar and Srinivasan deserve the credit for having the first available proof of Conjecture \ref{ordoneconj}. After this, we wrote the present proof of the higher order conjecture of Kumbhakar and Srinivasan in Section \ref{higherorderks}. We expect that readers who have closely followed the model theoretic developments of the last half-decade around binding groups (e.g. work of the first author with Jaoui, Jimenez and Moosa) can see that the proofs of the conjectures themselves follow relatively quickly from those developments once the connections are pointed out. Part of our goal in writing this note is to make those connections more transparent as well as pointing out some aspects of the optimality of the conjectures. During a Fall 2025 visit of Wei Li to Chicago, we obtained the connection to the work of Jaoui and Moosa contained in Section \ref{abredJM}. 

\subsection{Acknowledgment}  The authors thank Rahim Moosa for helpful comments on an earlier version of this manuscript.

\section{Notions from model theory} \label{prelim}
We begin by recalling some model-theoretic notions specialised to the theory DCF$_{0}$. We work inside a sufficiently saturated model $(\mathcal U,\delta)$ of DCF$_0$ and denote by $\mathcal C$ its subfield of constants. Let $F \supset K$ be an extension of differential fields. We write $a \ind _K F$ if $\tp (a/F)$ is a \emph{nonforking} extension of $\tp( a/K)$. This is equivalent to the equality of the Kolchin polynomials $\omega_{a/K} (t) = \omega_{a/F} (t)$ (see e.g. \cite{PongEmbedding}), which in turn is equivalent to the differential field $K\langle a\rangle$ being algebraically disjoint from $F$ over $K$. Sometimes we write $a \ind _K b$ for $a \ind_K K \langle b \rangle $. In this article, we mainly deal with elements $a$ which satisfy algebraic differential equations over $K$, and so $\omega _{a/K} (t)$ is equal to the transcendence degree of the field generated over $K$ by $a$ and its derivatives; i.e., the order of the minimal differential polynomial of $a$ over $K$. 

We say that two types $p, q $ over $K$ are \emph{weakly orthogonal} over $K$ if for all elements $a, b$ such that $\tp(a/K) = p $ and $\tp (b/K) = q$ we have  $a \ind _K b$. When $p$ and $q$ are weakly orthogonal over $K$, we write $p \perp ^w q.$ We say that two types $p, q$ over $K$ are \emph{orthogonal} if for all $F \supset K$, all non-forking extensions of $p, q$ to $F$, say denoted by $p',q'$ respectively, we have $p' \perp ^w q'$. When $p$ and $q$ are orthogonal, we write $p \perp q$. A sometimes useful perspective to keep in mind when considering the above notions is that nonorthogonality of types $p \not \perp q$ is equivalent to the weak non-orthogonality $p' \not \perp ^w q'$ of some nonforking extensions $p'$ and $q'$ of $p$ and $q$, respectively. 

When $X, Y$ are definable sets (or even type-definable sets) over $K$, we say that $X$ is internal to $Y$ if there is a definable surjection $f: Y^n \rightarrow X$ possibly using additional parameters to those used to defined $X$ and $Y$. The primary case of interest in this notion for this article takes place when $Y$ is taken to be the field of constants $\mathcal C$ and $X$ is taken to be the generic type of a differential equation (in one-variable) over $K$. 

When the definable surjection $f$ is defined over some $F \supset K$, internality is equivalent to there being $c_1, \ldots c_n$ in $Y$ such that the generic solution of $X$ can be written as a rational function of $c_1, \ldots , c_n$ over $F$. In this case, it follows from a general model-theoretic principle (see \cite{GST}, for instance) that the field $F$ 
%over which such an $f$ is defined 
can be generated by finitely many independent generic solutions of the equation $X$. When one replaces \emph{rational} over $F (c_1, \ldots , c_n)$ (or $f$ being definable) in the above definition with \emph{algebraic} over $F(c_1, \ldots , c_n)$ (or generic type of $X$ is algebraic over $F$ and a finite tuple from $Y$), we say that $X$ is \emph{almost internal} to $Y.$\footnote{There is considerable variation of the notation used in the literature around these notions, but we are generally following the notation used in \cite{freitag2023bounding}.} 

In comparing the notions of non-orthogonality and internality above, the reader will have noticed that both notions allow for an extension of parameters, while weak non-orthogonality is more sensitive and does not allow for such an extension. We will define next, the analog of weak non-orthogonality, but for the notion of internality - this is not a standard definition in model theory, but will be convenient for our purposes. Later in the paper, this notion will be related to the notion of an equation being of \emph{algebraic type}.

\begin{defn}

Let $X, Y$ be $K$-definable sets (or even partial types over $K$). 
%with $X$ internal to $Y$. 
We say that $X$ is \emph{$K$-definably internal to} $Y$ if there is a $K$-definable surjection $f: Y^n \rightarrow X$. We similarly define $X$ is \emph{$K$-definably almost internal to} $Y$ if any element of $X$ is algebraic over $K\langle c_1, \ldots c_d\rangle$ for finitely many realizations $c_1, \ldots c_d$ of $Y$. 
%is $K$-definable. 
\end{defn}

When a type $p$ over $K$ is $\mc C$-internal, the group of automorphisms of $p$ fixing $K$ and $\mc C,$ denoted $Aut(p/K \mc C)$ has the structure of a definable group, which model theorists call the binding group. This group of automorphisms is identical in this case to the group considered in Kolchin's strongly normal extensions. In general, this group acts definably on the type $p$; and in \cite{freitag2023bounding} the transitivity properties of this action were related to the problem of understanding weak non-orthogonality of the type $p$ to $\mc C$. 

\begin{defn} \cite[Definition 5.3]{freitag2023bounding}
    A definable action of a definable group $G$ of finite Morley rank on a definable set $S$ of finite Morley rank is \emph{generically $k$-transitive} if the diagonal action of $G$ on $S^k$ has an orbit $\mc O$ such that $RM(S^k \setminus \mc O) < RM(S^k)$. 
\end{defn}

For the rest of the section, let $G$ be the binding group of $X$ over $K$, where $X$ is the set of solutions of a $\mc C$-internal order $n$  differential equation. Let $p$ be the generic type of $X$ over $K$. In this case, the action of the group $G$ on $X$ is definably isomorphic %(over some differential field extension) 
to the constant points of an algebraic group action on an algebraic variety. 

\begin{lem} \label{gentransortho}
The action of $G$ on $X$ is generically $d$-transitive if and only if $p^{(d)}$ is weakly orthogonal to $\mc C$. 
\end{lem}
\begin{proof} The right-to-left direction is Proposition 2.3 of \cite{freitag2022any}. The left-to-right direction follows because if $p^{(d)}$ is not weakly orthogonal to $\mc C$, then there is a non-constant definable map $X^d \rightarrow \mc C$. The fibers of $f$ are (being $K (\mc C)$-definable) fixed by the action of $G$ and are dense in $X^d$.
\end{proof}

There is an a priori bound on the degree of generic transitivity that a group action can have, and this in turn leads to a bound on weak nonorthogonality of independent realizations of generic tuples of solutions (see \cite{freitag2023bounding, freitag2025finite} for more). This bound is a special case of the Borovik-Cherlin conjecture, a major open problem in the model theory of groups of finite Morley rank. 

\begin{thm} \cite{freitag2023bounding} \label{BCforACF0} When $RM(X) = d$ and $G$ acts generically $k$-transitively, then $k \leq d+2$. When $k=d+2$, the action $(G,X)$ is definably isomorphic to the natural action of $PSL_{d+1} (\mc C)$ on $\m P^d ( \mc C)$. 
\end{thm}

\section{Order one differential equations} \label{ordonesection}

We first repeat the classification of  order one differential equations given in \cite[Page 581]{kumbhakar2024classification}. Here $k$ is a differential field having field of constants $\mc C_k$. We will assume, in order to make certain statements simpler, that $k = k^{alg}$. In each of the results, this is not required, but if it is not assumed, at various places one may need to replace $k$ by a finite extension $\hat k$. This approach without the assumption that $k=k^{alg}$ is taken in \cite{kumbhakar2024classification}, but the approaches are easily seen to be equivalent. 

\begin{defn} \label{classification}
    Let $X$ be the solution set of $f(y,y')=0$, an irreducible differential equation of order one over $k$.
    \begin{itemize}
        
        \item[(I)] We say $X$ is of \emph{algebraic type} over $K$ if all solutions of $X$ in any differential field $F$ extending $k$ and having $\mc C_F = \mc C_k$ are algebraic over $k$. 
        
        \item[(II)] We say that $X$ is of \emph{Riccati type} if any generic solution $x$ of $X$, in a differential field extension with constants $\mc C_k$, is algebraic over a solution $u$ of a Riccati equation: $$x'=a_2 x^2 + a_1 x + a_0$$ with coefficients $a_i \in k=k^{alg}$ which are not all zero, that is $x\in k(u)^{alg}$. %{\color{purple} (In their paper, rather than using the larger field $K^{alg}$, they use a finite algebraic extension $\tilde{K}$. Shall we keep their original statement, or just use algebraic closure?){\color{blue}I think it's fine since for us $\tilde{K}=K(a_o,a_1,a_2)$.}}
        
        \item[(III)] We say that $X$ is of \emph{Weierstrass type} if any generic solution $x$ of $X$, in a differential field extension with constants $\mc C_k$, is algebraic over a solution $u$ of a Weierstrass equation:
        $$(x')^2 = \alpha ^2 (4x^3 - g_2 x -g_3)$$ where $g_2, g_3 \in \mc C_k$ such that $27g_3^2 - g_2^3 \neq 0$ and $\alpha \in k=k^{alg}$ is nonzero, that is $x\in k(u)^{alg}$. 
        \item[(IV)] The differential equation $X$ is of \emph{general type} if it is not of one of the above types. 
    \end{itemize}
\end{defn}

%{\color{blue}(OLS: are we assuming from now on that $X$ is an order one equation? If so, we should say so here.)}

Unless otherwise stated, in this section $X$ denotes the solution set of an irreducible order one differential equation $f(y,y')=0$ over $k$. Note that then $X$ is strongly minimal. We now provide a characterization of being of algebraic type.

\begin{lem} \label{equivofnotionsinorderone}
    %Let $f(y,y')=0$ be an order one (irreducible) differential equation  over $k$ and let $X$ be its set of solutions. Then, 
    The following are equivalent:
    \begin{enumerate}
    \item $X$ is of algebraic type.
        \item The generic type of $X$ is not weakly orthogonal to $\mc C$ (over $k$).
        \item $X$ is $k$-definably almost internal to $\mc C$.
        \item There exists a differential rational map $g:X\to \mathcal C$ over $k$ whose image is dominant.
    \end{enumerate}
\end{lem}

\begin{proof}
Let $a$ be a generic solution of $X$. The generic type of $X$ is weakly non-orthogonal to $\mc C$ over $k$ if and only if there exist 
%a generic solution $x$ of $X$ and 
a constant $c$ in $\mc C_{k\langle a\rangle^{alg}}\backslash\mc C_k$; in other words, $k\langle a\rangle^{alg}=k(a)^{alg}=K(c)^{alg}$. Note that the latter is equivalent to $X$ being $k$-definably almost internal to $\mc C$. This shows $(2)\Leftrightarrow (3)$. Furthermore, as any transcendental solution is generic, the equality $k(a)^{alg}=k(c)^{alg}$ is equivalent to $X$ being of algebraic type, showing $(2)\Leftrightarrow (1)$. For $(2)\Leftrightarrow (4)$, simply note that, since $\mathcal C_k$ is algebraically closed (as we are assuming $k$ is), if there is $c\in \mc C_{k\langle a\rangle^{alg}}\backslash\mc C_k$, then there must be $d\in \mc C_{k\langle a\rangle}\backslash\mc C_k$, and this induces a differential rational function $X\to \mathcal C$ over $k$.

\begin{comment}
Note that $X$ being of algebraic type is equivalent to the function field of $f$, $K_f=\text{Frac}(K[y,y']/(f(y,y')))$, which is differentially isomorphic to $K\langle x\rangle$ for $x$ any generic solution of $X$, contains a new constant\footnote{This result was first proved in \cite[Theorem 2.8]{vanderput2015}. Here, we directly use our model theoretic language to give a simple proof.}. 
Indeed, if $X$ is of algebraic type, as any generic solution is a transcedental solution, the function field $K_f$ of $f$ contains a new constant; on the other hand, if $\mathcal C_{K_f}\neq \mathcal C_K$, then $X$
is weakly nonorthogonal to $\mathcal C$ over $K$. 
Let $a\in X(F)$ with $F$ a differential field extension of $K$ with $C_F=C_K$. %If $a\in X(K^{alg})$, then we are done. 
For contradiction, assume that $a$ is $K$-generic.  By the weakly non-orthogonality, there exists $c\in \mc C_{K\langle a\rangle^{alg}}\backslash\mc C_K$ with $K(c)^{alg}=K(a)^{alg}$. So $c\in \mathcal C_{K(a)^{alg}}\subset  \mathcal C_{F^{alg}}=\mathcal C_K$. This is a contradiction and thus, $a\in X(K^{alg})$ and $X$ is of algebraic type.   
As $\mc C_K$ is algebraically closed, for a generic solution  $x\in X(\mathcal U)$, $C_{K\langle x\rangle^{alg}}=(C_{K\langle x\rangle})^{alg}$, so  $C_{K\langle x\rangle}\backslash\mc C_K\neq \emptyset$  is equivalent to $C_{K\langle x\rangle^{alg}}\backslash\mc C_K\neq \emptyset$.
\end{comment}
\end{proof}

%{\color{red} OLS: The previous lemma had a longer proof. I commented it out. Also, there was another lemma that I incorporated as point (4) in the above lemma (the old lemma is in the file, just commented out.}

\begin{comment} 
\begin{lem} \label{weakorthogonewC}
Weak non-orthogonality of the generic type of $X$ to $\mc C$ over $K$ is equivalent to the existence of a $K$-differential rational map $g: X \rightarrow \mc C$ over $K$ whose image is dominant. 
\end{lem}

\begin{proof} This argument is essentially contained in the second paragraph of \cite[Proposition 2.3]{freitag2022any}. 
Let $p$ be the generic type of $X$. It follows by using elimination of imaginaries for the structure $\mc C$ that for $a \models p$, there is a constant $c \in \mc C$ such that
$c \in dcl (Ka) \setminus acl(K)$, but as $a \in X$ is generic over $K$, this induces a nonconstant $K$ differential rational map $X \rightarrow \mc C$.
\end{proof}
\end{comment}

Conjecture~\ref{ordoneconj} will follow from the following theorem (or rather its proof).
\begin{thm} \label{characterizinggentype}
Let $p$ be the generic type of $X$ over $k$. The following are equivalent:
\begin{enumerate}
\item $X$ is of general type. 
\item the differential field generated by any tuple of  realisations of $p$ has constants $\mc C_k$. 
%{(\color{red} Do we require independent realisations of $p$ ?)}
\item $p \perp \mc C$. %where $p$ is the type of a generic solution to $X$.  
\item $p$ is trivial (in the model theoretic sense).
\end{enumerate}
\end{thm}
\begin{proof}
For $(3)\Rightarrow (2)$, note that if $a_1,\dots,a_n$ are realisations of $p$, then, since $X$ is of order one, there is a sub-sequence $b_1,\dots,b_m$ such that the $b_i$'s are independent realisations and $acl(a_1,\dots,a_n)=acl(b_1,\dots,b_m)$. But since $p\perp \mc C_k$, we have $p^{(m)}\perp ^w \mc C_k$; in particular, $acl(b_1,\dots,b_m)\cap \mc C=\mc C_k$, which implies (2). On the other hand, if $p \not \perp \mc C$, by say \cite[Lemma 4.3.1]{GST} or \cite{hrushovski1989almost}, there is (finite) tuple $a=(a_1,\dots,a_n)$ of (independent) realisations of $p$ such that $a\not \perp ^w \mc C$ 
%{\color{red} (independent realisations of $p$ witnessing $p^{(n)}\not \perp ^w \mc C$?)}
; in other words the differential field generated by $a$ over $k$ has a constant not in $\mc C_k$.

Since $p$ is order one, the classification of strongly minimal types in DCF$_0$ \cite{HrushovskiSokolovic} implies $(3)\Leftrightarrow (4)$. 

In case $X$ is not of general type, then $p$ must be almost internal to $\mathcal C$ (and hence nonorthogonal to $\mc C$). Indeed, if $X$ is of algebraic type we deploy Lemma~\ref{equivofnotionsinorderone}, and if $X$ is of Ricatti or Weierstrass type then $p$ is strongly normal in the Kolchin sense by \cite{kolchin1953galois}, and hence internal to the constants.

Now assume $p$ is nonorthogonal to $\mc C$. If $p \not \perp ^w \mc C$, then Lemma \ref{equivofnotionsinorderone} yields that $X$ is of algebraic type. Thus we can further assume that $p \perp ^w \mc C$. Then using a similar argument as in the proof of Lemma \ref{equivofnotionsinorderone}, we get $p$ is almost internal to $\mc C$. Using Lemma 3.6 of \cite{jin2020internality} we get that $p$ is interalgebraic with a type which is $\mc C$-internal, but the cases of II and III in Definition \ref{classification} of non-general type equations allow for interalgebraicity, so we can replace $p$ with this interalgebraic type. So, without loss of generality, we assume that $p$ is itself $\mc C$-internal. 

It is well-known that in this case ($p$ is $\mc C$-internal and weakly $\mc C$-orthogonal) the action of the binding group $G$ on $X$ is transitive - this follows, for instance from \cite[Lemma 2.2]{freitag2023bounding}. From this point, we use the classification of group actions on strongly minimal sets, where each case essentially involves finding a differential rational\footnote{The expression is even differential birational as the expression is interdefinable with the generic solution of $X$.} expression of the generic solution of $X$ which satisfies an equation of a very particular form relating to the group $G$, what Kolchin \cite[VI.9]{KolchinDAAG} calls the a full logarithmic-differential equation of $G$ over $K$. It is a result of Hrushovski \cite{hrushovski1989almost} that the action of $G$ on $X$ is definably isomorphic to one of five group actions: 
\begin{enumerate}
    \item The action of $\m G_a ( \mc C) $ on $\m A^1 ( \mc C)$ by translation. 
    \item The action of $\m G _m ( \mc C) $ on $\m A^1 ( \mc C ) \setminus \{0 \}$ by multiplication. 
    \item The action of the affine group $Aff_1 ( \mc C)$ %{\color{red}OLS: Shouldn't this be $Aff_1(\mathcal C)$, unless the notation differs from standard notation, note that $Aff_n$ has dimension $n^2+n$} 
    on $\m A^1 ( \mc C)$ by affine transformation. 
    \item The action of $PGL_2 ( \mc C) $ on $\m P^1 ( \mc C)$ by linear fractional transformation. 
    \item The regular action of an elliptic curve on itself. 
\end{enumerate}

The details of each of the cases might be found in various places, but Proposition 6.5 of \cite{jaoui2025abelian} gives the details of the first four cases, which correspond to Ricatti type. The details of the last case, which correspond to Weierstrass type, can be found in \cite[\S III.8]{kolchin1953galois} after noting that 
%by Fact 6.4 of \cite{jaoui2025abelian} 
since the binding group is commutative, $p$ is itself fundamental (i.e., the action of the binding group is regular, as is the case for any faithful transitive commutative group action). Hence $X$ is not of general type.
\end{proof}

We note that there are numerous potential paths to proving the previous theorem. For instance, it seems likely that the result of Jaoui and Moosa \cite{jaoui2025abelian} was classically known in many cases, but we don't know another reference which covers our specific situation. 

\begin{rem}
From Theorem \ref{characterizinggentype} alone, one can deduce, for instance that any order one equation whose generic type is orthogonal to $\mc C$ provides a positive answer to the question of \cite[beginning of section 7]{noordman2022autonomous}; namely, is there a $D$-algebraic element which is not $D^n$-finite for any $n\geq 1$? See also \cite[Section 1.2]{kumbhakar2024classification}, where the question is discussed. %{\color{blue}(OLS: should we add the question?)}. 
The observation of \cite{kumbhakar2024classification} uses the difference in the structure of algebraic relations between solutions of an equation of general type and a $D$-finite equation, an idea also exploited in \cite[section 7]{noordman2022autonomous} to give an example of a $D$-algebraic, but not $D$-finite function. The examples of \cite{noordman2022autonomous} and  \cite{kumbhakar2024classification} are order one equations of general type (e.g.  \cite{noordman2022autonomous} uses an example given by Rosenlicht \cite{notmin}). Additional very early examples of equations of general type can be found in \cite{shelah1973differentially} and \cite{kolchin1974constrained}, where they were used in a similar capacity to that of \cite{notmin}. Each of these equations has the property that the type of a generic solution over the field of definition is orthogonal to the constants - as we have seen in order one, this property corresponds to being of general type. Many additional examples of higher order equations have now been shown to be orthogonal to $\mc C$. For instance, the $j$-function \cite{freitag2017strong}, more generally uniformizing functions of genus zero Fuchsian groups \cite{casale2020ax}, and 
Painlev\'e equations \cite{nagloo2017algebraic} each give examples of higher order equations answering the question posed in \cite{noordman2022autonomous}.   
\end{rem}

We now prove Conjecture \ref{ordoneconj}:
\begin{thm}\label{proofconj}
   An order one differential equation over $k$ (respectively, an autonomous differential equation over $\mc C_k$) is not of general type if and only if it has at most three (respectively, at most one) algebraically independent solutions in any given differential field extension of $k$ (respectively, $\mc C_k$) having $\mc C_k$ as its field of constants. 
\end{thm}

\begin{proof} 
If $X$ is of general type, then Theorem~\ref{characterizinggentype} implies that any tuple (of any length) of independent realisations of $p$ will generate a differential subfield with constants $\mc C_k$. 

On the other hand, if $X$ is not of general type, the proof of Theorem~\ref{characterizinggentype} yields that the number of independent realisations in any differential field extension with constants $\mc C_k$ is bounded by the dimension of the binding group and that this dimension is bounded by three (by going over the 5 cases displayed in that proof). Moreover, when $k=\mc C_k$, by Fact~6.1 of \cite{jaoui2025abelian} we get that the binding group is commutative and hence it has dimension one (by looking at cases (1), (2) and (5) in the proof of Theorem~\ref{characterizinggentype}).

%Every order one equation is either almost $\mc C$-internal or is $\mc C$-orthogonal. The result then follows from Corollary \ref{almostint} in the almost internal case together with Theorem \ref{characterizinggentype} and Lemma \ref{gentp} in the $\mc C$-orthogonal case. 
\end{proof}

\subsection{Improvements to the conjecture}
There are two natural ways in which we see how one might improve the conjecture. The first is a refinement in the case of general type equations, while the second extends the conjecture to the case of partial differential equations. 

More can be said with regard to the general type case of Conjecture \ref{ordoneconj} using a result of Hrushovski (cf. \cite[Corollary 1.82]{pillay2002}) which was latter generalized by Freitag and Moosa to the partial differential setting (cf. \cite[Theorem 6.2]{freitag2017finiteness}):

\begin{thm} \label{catone}
%Suppose $X$ is a strongly minimal order one definable set that is orthogonal to $\mc C$. Then $X$ is $\aleph_0$-categorical. 
If $X$ is of general type, then $X$ is $\aleph_0$-categorical.
\end{thm}

A restatement of the theorem in purely algebraic terms is as follows: The condition of $\aleph_0$-categoricity in this setting is equivalent to there being a specific fixed (algebraic) polynomial $p(x,y)$ over $k$ such that if $a_1, \ldots , a_n$ are solutions of $X$ which are algebraically dependent over $k$, then for some pair of solutions $a_i, a_j$, we have $p(a_i,a_j)=0$. The equivalence relation for solutions of $X$ which are zeros of $p$ is then a $k$-definable equivalence relation, and so by elimination of imaginaries, there is a surjective differential rational map $g:X \rightarrow Y$ over $k$ for some order one differential equation $Y$ such that two elements of $X$ are algebraically dependent if and only if they have the same image in $Y$. The equation $Y$ has the property that if $b_1, \ldots , b_n$ are distinct solutions of $Y$ which are not algebraic over $k$, then $b_1, \ldots , b_n$ are algebraically independent over $k$. 

%{\color{red} OLS: I added the following corollary.} 
We can now state and prove the version of Conjecture~\ref{ordoneconj} that we suggested in the introduction as the \emph{existential version} of the main result of \cite{freitag2022any}.

\begin{cor}
    Let $k^{\diff}$ be the differential closure of $k$. If there exists four algebraically independent solutions of $X$ in $k^{\diff}$, then there exists infinitely many algebraically independent solutions in $k^{\diff}$. When $k=\mc C_k$, four can be replaced by two.
\end{cor}
\begin{proof}
    Suppose there are more than four independent solutions in $k^{\diff}$. Then, Theorem~\ref{proofconj} implies that $X$ is of general type. Theorem~\ref{catone} then implies that $X$ is $\aleph_0$-categorical, and by the comments above we can find $Y$ defined by an order one differential equation such that if $b_1, \ldots, b_n$ are distinct solutions of $Y$ which are not algebraic over $k$, then $b_1, \ldots , b_n$ are algebraically independent over $k$. Since $k^{\diff}$ is differentially closed but not equal to $k$, we can find arbitrarily many distinct solution in $Y(k^{\diff})$ that are not in $k=k^{alg}$.
\end{proof}

None of the results used in the model-theoretic proof of Conjecture \ref{ordoneconj} or the extension of the above paragraphs depend closely on working in the setting of ordinary differential equations. Each of the results remains true in the case of partial differential equations with commuting derivations in which the constant field $\mc C$ is taken to be the intersection of the field of constants of each of the commuting derivations. Here the relevant model theory is that of the theory DCF$_{0,m}$ \cite{McGrail}. 

The results of \cite{freitag2017finiteness}, \cite{freitag2023bounding}, and \cite{freitag2025finite} each apply in this more general setting, yielding the same results when the hypothesis is that the $\Delta = \{\delta_1, \ldots , \delta_m \}$-function field generated by the generic solution of $X$ is of transcendence degree one. There is a close dependence on the transcendence degree of this function field being one - as we show later in the paper, results for higher order equations must necessarily be weaker.

\begin{thm} \label{orderonecategoricityresult}
   Let $X$ be the set of solutions of a system of $\Delta$-equations over an algebraically closed $\Delta$-field $k$ such that if $a$ is a generic solution to the system then $\text{tr.deg}_k k \langle a \rangle=1$. Let $\mc C_k=\mc C_k^{alg}$ be the kernel of $\Delta$. Then, either 
   \begin{itemize}
       \item $X$ is almost $\mc C$-internal and has at most three algebraically independent solutions in any $\Delta$-field $F$ extending $K$ such that $\mc C_F = \mc C_k$, or 
       \item there is a $\Delta$-field extension $F$, with $\mc C_F = \mc C_k$, containing infinitely many algebraically independent solutions of $X$.
       %over $K$ in fields $F$ with $\mc C_F = \mc C_k$.
   \end{itemize} 
In the second case, any collection of solutions $a_1, \ldots , a_n$ of $X$ satisfies an algebraic relation over $k$ if and only if some pair $a_i, a_j$ satisfies an algebraic relation over $k$. In fact, there is a single polynomial $p(x,y)$ depending on $X$ such that when the pair satisfies an algebraic relation over $k$, we must have $p(a_i, a_j)=0$. It follows that the $\Delta$-field $F$ can be taken to be the $\Delta$-closure of $k$.
\end{thm}

\section{On the optimality of the conjecture of Kumbhakar, Roy, and Srinivasan} \label{orderoneoptimal}
In Conjecture \ref{ordoneconj} above, the bound three is replaced by one in the autonomous case (the case in which the equation has constant coefficients). One might inquire: can the bound of three be improved for the class of order one differential equations over \emph{any differential field other than a constant field?} We show that the answer is no, at least when the differential field $k$
%has finite transcendence degree 
is differentially finitely generated over its field of constants. 

\begin{rem}
    We note that when $k$ is PV-closed the bound in Conjecture~\ref{ordoneconj} is at most one, as in this case being of nongeneral type implies being of algebraic or Weierstrass type. Similarly, when $k$ is differentially closed the bound is at most zero. Of course, in either of these cases $k$ is not differentially finitely generated over its constants.
\end{rem}

Over any differential field $k$ with algebraically closed field of constants such that there is an order two linear equation $L(x)=0$ whose differential Galois group is $GL_2$ (or $SL_2$), the associated Riccati equation (image of the $L(x)=0$ under the map $x \mapsto x'/x$) has binding group $PGL_2=PSL_2$ and the action of the binding group on this Riccati equation is isomorphic to the action of $PGL_2$ acting on $\m P ^1$. This action is $3$-transitive, and so any $3$ solutions are independent, however the equation is not of general type. Any fourth solution is given by a rational expression of the first three and an arbitrary constant, see \cite{nagloo2019algebraic} or \cite[4.3]{freitag2023bounding}.

Mitschi and Singer \cite{MitschiSinger1996} solved the inverse problem for differential Galois theory for connected linear algebraic groups over any nonconstant differential field of finite transcendence degree, proving: 

\begin{thm}
    Let $C$ be an algebraically closed field of characteristic zero, $G$ a connected linear algebraic group over $C$, and $k$ a differential field containing $C$ as its field of constants and of finite nonzero transcendence degree over $C$. Then, $G$ can be realized as the Galois group of a Picard-Vessiot extension of $k$. 
\end{thm}

In Appendix A, Singer generalizes this theorem to the case in which the field $k$ is finitely generated as a differential field over its field of constants. We only need to apply this result in the case of $GL_2$ or $SL_2$ to see: 

\begin{prop}\label{useinverse}%\deleted{and of finite, nonzero transcendence degree over $C$}.
    Let $C$ be an algebraically closed field of characteristic zero and $k$ a {nonconstant} differential field containing $C$ as its field of constants. Assume that $k$ {is finitely generated over $C$ in the differential sense.} Then Conjecture \ref{ordoneconj} is not true with two in place of three.
    %{(As we now have the Appendix, we can change $k$ to be a differential field  containing $C$ as its field of constants that is finitely generated over $C$ in the differential sense.) }
\end{prop}

%{\color{blue} OLS: So the statement above should be: Suppose $K$ is not a constant field, its constant subfield $C$ is algebraically closed, and $K$ is differentially finitely generated over $C$. Then Conjecture \ref{ordoneconj} is not true with two in place of three.)}

\begin{proof}

Let $\delta X = BX$, $B \in GL_2 (k)$, be a system of linear differential equations over $k$ with binding group $GL_2 (\mc C)$ or $SL_2 ( \mc C)$. For any $k$ as above, such a matrix exists by Appendix~A,  Theorem \ref{thm-A}. By doing gauge transformations over $k$, we may assume without loss of generality that the matrix $B$ is of the form $$
\begin{pmatrix}
0 & 1 \\
a_1 & a_2
\end{pmatrix}
$$
Then $y_2'/y_2$ satisfies a Riccati equation (see Section \ref{thegeneral} or Section 4 of \cite{freitag2023bounding}) whose binding group is $PGL_2 ( \mc C)$. Let $p$ be the generic type of this equation. As this group acts transitively and generically $3$-transitively on $p$, we can find three algebraically independent solutions in the differential closure of $k$.
%$p^{(3)}$ is weakly orthogonal to $\mc C$ and $p^3 = p^{(3)}$. Any three solutions of the equation in the differential closure of $k$ are algebraically independent over $k$.  
\end{proof}

The strong classification of differential equations which are not almost $\mc C$-internal depends closely on the order one assumption of the previous section. Here, we mean the strong local finiteness result for order one equations of general type of Theorem \ref{catone} and stated in algebraic terms following that theorem. A counterexample appears already for order two equations: 
$$\frac{-y}{(x-t)^2} + \left( 2t(t-1)\frac{x'}{y} \right) ' \frac{t(t-1)x'}{(x-t)y} = 0, $$
where $y^2 = x(x-1)(x-t).$ This formula for the Manin kernel of the elliptic curve $E_t$ can be found in appendix B.3 of \cite{PongThesis}. The equation is orthogonal to $\mc C$, but the graphs of multiplication-by-$n$ for $n \in \m N$ in elliptic curve provide algebraic relations of unbounded degree between generic solutions of the equation. By expanding the above equation and rewriting it in terms of the variable $x$, one obtains the standard form of the sixth Painlev\'e equation $P_{VI}(0,0,0,1/2)$ (cf. \cite[Theorem 1.1]{Manin6}). An analogous equation, which we call the Manin Kernel, can be defined on any abelian variety (with similar properties when $A$ is simple and not isomorphic to an abelian variety over $\mc C$). Hrushovski and Sokolovic \cite{HrushovskiSokolovic} proved that every strongly minimal set which is orthogonal to $\mc C$ and is not \emph{geometrically trivial} is nonorthogonal to a Manin Kernel. This led to the question of Hrushovski \cite{hrushovski1998geometric}: Is every geometrically trivial strongly minimal set (which one might have taken to be the definition of general type for higher order equations) is $\aleph_0$-categorical? This was refuted in \cite{freitag2017strong}, where the differential equation satisfied by the $j$-function provides a counterexample. There, isogenies of elliptic curves provide the algebraic relations between generic solutions whose degrees are unbounded. A revision of the question of Hrushovski would be of great interest. To date, one has not been made. On the other hand, the work in \cite{casale2020ax} provides a suggestion. It is shown there that all instances of non $\aleph_0$-categorical trivial strongly minimal equations are related to arithmetic groups. It is an interesting open problem to find such a strongly minimal set that is unrelated to arithmetic groups.

\section{Higher order equations - the Kumbhakar-Srinivasan conjecture} \label{higherorderks}

We continue with an algebraically closed differential field $(k,\delta)$ of characteristic zero. We first recall the notation of \cite{kumbhakar2024classification}, which was remarked on in the introduction. 
%Let $k$ be algebraically closed differential field of characteristic zero. 
Let
\begin{equation} \label{arb} f(y, y' , \ldots y^{(n)}) = 0 
\end{equation} be an algebraic differential equation of order $n$, where $f$ is an irreducible polynomial over $k$. Let $y$ be a nonalgebraic solution to the equation (\ref{arb}). Let $\mf S_{k,f,y}$ be the set of solutions of (\ref{arb}) satisfying: 
\begin{enumerate} 
\item For any $y_1 \in \mf S_{k,f,y}$, there is a differential $k$-isomorphism between the differential fields $k \langle y \rangle$ and $k \langle y_1 \rangle $ that maps $y$ to $y_1$. 
\item For any distinct $y_1, \ldots , y_m \in \mf S_{k,f,y}$, the field of constants of $k \langle y_1 , \ldots , y_m \rangle$ is $C$ and $\text{tr.deg}_k k \langle y_1 , \ldots , y_m \rangle = \sum _{i=1} ^m \text{tr.deg}_k k \langle y_i \rangle $.
%$C=\mc C(k)$. 
%\item For any $y_1 , \ldots , y_2 , \ldots , y_m \in \mf S_{k,f,y}$, $\trdeg _k (k \langle y_1 , \ldots , y_m \rangle) = \sum _{i=1} ^m \trdeg _k (k \langle y_i \rangle ).$
\end{enumerate}  %{\color{blue} OLS: here the set of solutions are elements in a monster model such that the differential field generated by a solution has no-new-constants... should we say this?}

We next give a proof of Theorem \ref{generalizethis} for all $n \in \m N$. Our theorem establishes a stronger bound when the field $k$ is assumed to be a field of constants.

\begin{thm} \label{thegeneral}
%Let $k$ be an algebraically closed differential field  of characteristic zero and 
Let $y$ be a nonalgebraic solution of (\ref{arb}). Then, $\mf S_{k,f,y}$ is a finite set if and only if it has at most $n+2$ elements. When the field $k$ is a field of constants, the bound $n+2$ can be replaced by $1.$  %{\color{red}( I am puzzled: Should 2 here be replaced by 1? We claimed  in the proof that $p^{(2)}$ is not weakly orthogonal to $\mc C$.  ---Wei)}
\end{thm}

\begin{proof}
%Since the bound specified in part (3) of Theorem \ref{generalizethis} is increasing in the order of the equation, and we will prove that it holds for arbitrary equations, 
It is sufficient to prove the bound when restricting to the complete type $p$ given by the generic solution of equation (\ref{arb}) rather than an arbitrary non-algebraic solution of (\ref{arb}) as in the statement, as  non-algebraic solutions that are not generic satisfy lower-order differential equations. 

If $p$ is orthogonal to the constants, then taking $\{a_i \}_{i\in \m N}$ to be an independent sequence of realizations of $p$ shows that $\mf S_{k,f,y}$ is infinite for the generic solution $y$ of (\ref{arb}).

On the other hand, if $p$ is nonorthogonal to the constants, then $p^{(m)}$ is not weakly orthogonal to $\mathcal{C}$ for some $m \in \m N$.  This follows 
%by stable embeddedness, 
for instance from \cite[Lemma 4.3.1]{GST}, but maybe the first place where this is specifically written is \cite{hrushovski1989almost}\footnote{Note that the terminology was evolving at the time, and what we call weak orthogonality was called almost orthogonality in \cite{hrushovski1989almost}. We prefer to use the term ``almost" in a different way - often in conjunction with internality statements - usually modifying statements in a way which is up to finite index.}, where effective variants of this result (that is, bounding $n$ in terms of some data) were considered for the first time. But in fact, in \cite{freitag2025finite}, Corollary 5.1, the authors show that when $p$ is nonorthogonal to the constants, $p^{(d+3)}$ is already not weakly orthogonal to $\mc C$ where $d$ is the Lascar rank $U(p)$. Noting that $U(p)$ is bounded above by the order of the equation \ref{arb} (see for instance \cite{MMP} or \cite{PongRank}), it follows that $p^{(n+3)}$ is not weakly orthogonal to $\mc C$ (another way to see this is by using Lemma~\ref{gentransortho} and Theorem~\ref{BCforACF0} from \S\ref{prelim}). It is easy to check that the latter implies $|\mf S_{k,f,y}|\leq n+2$. In the case when $k=\mc C_k$, \cite[Theorem 1.1]{jaoui2025abelian} implies that $p^{(2)}$ is not weakly orthogonal to $\mc C$ and hence it follows that in this case $|\mf S_{k,f,y}|\leq 1$.

\begin{comment}
We have:

\begin{lem} \label{5.4}
Let $y$ be a generic solution of (\ref{arb}). Then $\mf S_{k,f,y}$ is infinite if and only if $p = \tp (y/k)$ is orthogonal to the constants. More specifically, $|\mf S_{k,f,y}| = d$ if an only if $p^{(d+1)}$ is not weakly orthogonal to $\mc C$ and $p^{(d)}$ is weakly orthogonal to $\mc C$. 
\end{lem}

In \cite{freitag2025finite}, Corollary 5.1, the authors show that when $p$ is nonorthogonal to a definable set $X$, it follows that $p^{(d+3)}$ is not weakly orthogonal to $X$ where $d$ is the Lascar rank $RU(p)$. Noting that $RU(p)$ is bounded by the order of the equation \ref{arb} (see for instance \cite{MMP} or \cite{PongRank}), it follws that if $p$ is nonorthogonal to $\mc C$, then $|\mf S_{k,f,y}|\leq n+2.$ 

When the field $k$ is a field of constants, \cite[Theorem 1.1]{jaoui2025abelian} implies that $p^{(2)}$ is not weakly orthogonal to $\mc C$. Again, Lemma \ref{5.4} applies to yield the result in this case.
\end{comment}
\end{proof}

% In \cite[Theorem 2.8]{vanderput2015}, for a first order differential equation $f(y,y')=0$, the authors proved that $f$ is of algebraic type if and only if the function field of $f$ contains a new constant.  Their proof can be easily adapted to show that the  differential equation $f(y,y',\ldots,y^{(n)})=0$ over $K$ does not have a generic solution in any differential field extension $F$ of $K$ with $\mc C_F=\mc C_K$  if and only if the differential function field of $f$,
% $K\langle f\rangle=Frac(K[y,y',\ldots,y^{(n)}]/(f))$, contains a new constant. 
%On the other hand,  we can characterize $f(y,y',\ldots,y^{(n)})=0$ being of algebraic type by showing the function field contains ``enough" independent new constants:

%\begin{prop}
%All solutions of $f(y,y',\ldots,y^{(n)})=0$ taken from differential field extensions of $K$ with no new constants are algebraic if and only if $\trdeg\, \mc C_{K\langle f\rangle}/\mc C_K=n$.
%\end{prop}
%{\color{purple} \begin{proof}
% This proposition still lacks a proof. 
% Is it valid? 
%\end{proof} }

\subsection{Some remarks on the proof of Theorem \ref{thegeneral}.} \label{explainhigherorderproof}

One might notice that the result of \cite{kumbhakar2025strongly} seems to rely on classification results for algebraic group actions, which we have not explicitly appeared in our proof above (e.g. the results of \cite{Kaga, popov1973classification, popov1975classification, cantat2018algebraic}). The uses of these results build towards understanding the relationship between an algebraic group $G$ on an algebraic variety $V$ in various terms and situations, but typically understanding the extremal cases in which $V$ is a minimal homogeneous spaces of $G$ as in \cite[Section 8]{cantat2018algebraic} and $G$ is a differential Galois group. 

Any appearance of the absence of such specific results in our above proof is a red herring; the result of \cite{freitag2025finite} we apply above relies on the proof of the Borovik-Cherlin conjecture, which is of a similar (but more general) flavor as the above results. 
%The conjecture is most commonly stated in terms of \emph{generic transitivity}, but we give the equivalent formulation in terms of bounding the dimension of the group $G$, which implies the necessary bounds on the minimal homogeneous spaces of given groups. 
For algebraic groups in characteristic zero, the conjecture was proved in \cite{freitag2023bounding}. 
%in which it is shown that for a faithful primitive action of an algebraic group $G$ on an algebraic variety $V$, the dimension of $G$ is at most $\dim (V) (\dim (V) +2).$ 
This was generalized to differential algebraic group actions on differential algebraic varieties in \cite{freitag2025finite}. An examination in the work of Kumbhakar and Srinivasan \cite[\S5]{kumbhakar2025strongly} shows that the key bounds they use follow from a special case of the Borovik-Cherlin conjecture, namely when $\dim V\leq 4$ (which is the case they are able to prove). We note that this special case follows from the work of \cite{freitag2023bounding}.

\section{On the optimality of the conjecture of Kumbhakar and Srinivasan} \label{optimalks}

In the previous section, we showed in Theorem \ref{thegeneral} that a stronger form of the conjecture of Kumbhakar and Srinivasan holds when $k$ is a field of constants - that is, the bound is one rather than $n+2$. It is natural to ask if this strong bound holds over any other %finitely generated  
differential field. The following result gives a negative answer. The proof follows the same lines as the proof of Proposition~\ref{useinverse}.

\begin{prop}
    Let $C$ be an algebraically closed field of characteristic zero and $k$ a %finitely generated 
    nonconstant differential field with field of constants $C$ and that is differentially finitely generated over $C$ 
    %in the differential sense. 
    Then, Theorem~\ref{thegeneral} does not hold with $n+1$ instead of $n+2$.
    %there is a differential equation over $k$ such that the bound of $n+2$ in Theorem \ref{thegeneral} is tight. 
\end{prop}

\begin{proof}
Let $\delta X = BX$, $B \in GL_n (k)$, be a system of linear differential equations over $k$ with binding group $GL_n (\mc C)$ or $SL_n ( \mc C)$. For any $k$ as above, such a matrix exists by Appendix A, Theorem \ref{thm-A}. Via appropriate gauge transformations over $k$, without loss of generality that the matrix $B$ is of the form $$
\begin{pmatrix}
0 & 1 & 0 & \ldots &  0 & 0 \\
0 & 0 & 1 & 0 & \ldots  & 0\\
& & & \vdots & & \\
0 & 0 & 0 & 0 & \ldots  & 1 \\
-a_0 & -a_1 & -a_2 & -a_3 & \ldots  & -a_{n-1} \\
\end{pmatrix} \begin{pmatrix} y_0 \\ 
y_1 \\
y_2 \\
\vdots \\
y_{n-1}
\end{pmatrix}
$$
at which point solutions of the resulting scalar order $n$ equation $y^{(n)}=\sum_{i=0} ^{n-1} a_i y^{(i)}$ are in bijection with the solutions of the matrix equation with $y_i=y^{(i)}$. 

Then $y' / y$ satisfies a nonlinear equation of order $n-1$ which we refer to as the associated Riccati equation, see either \cite{freitag2023bounding} or \cite{kumbhakar2024classification}. The equation is of the form

$$P_{d} (y)  + a_{d-1} P_{d-1} (y)  + \ldots + a_1 P_1(y) + a_0 P_0 (y),$$
where $P_0=1$ and $P_i = \delta (P_{i-1}) + y P_{i-1}.$ 

The binding group of the associated Riccati equation is $PGL_n ( \mc C)$. Let $p$ be the generic type of this equation. As this group acts transitively and generically $(n+2)$-transitively on $p$, $p^{(n+2)}$ is weakly orthogonal to $\mc C$. Taking $f$ to be the minimal differential polynomial isolating $p$, we see that $\mf S_{k,f,y}$ is of size at least $n+2$ (as $p$ has $n+2$ independent realisations that generate a differential field with constants $\mc C_k$). 

%and the type is isolated by the vanishing of a differential equation $f$ (which can be easily calculated as explained in the next section). Taking $n+2$ many generic independent realizations of $p$ over $k$, we can see that $\mf S_{k,f,y}$ is of size at least $n+2.$
\end{proof}

There are numerous possible variants of the conjecture of Kumbhakar and Srinivasan. We have mentioned the case of partial differential equations above in Theorem~\ref{orderonecategoricityresult}. One might also consider more than one differential equation at a time, allowing the elements to come from the solutions sets of several differential equations. This problem is considered in \cite{freitag2025finite} in the language of nonorthogonality. In more general model theoretic terms, the problem is discussed in Section 2 of \cite{hrushovski1989almost}. In \cite{freitag2025finite}, the authors show generally that when $p$ and $q$ are the generic types of the zero set of an order $n$ and $m$ differential equation, respectively, with the property that $p \not \perp q$, then $p^{(n+3)} \not \perp ^w q ^{(m+3)}.$ Jimenez \cite{jimenez2024note} shows that there are examples which show that the bounds $(n+3, m+3)$ can not be improved. Can one improve the bounds on nonorthogonality when replacing $p^{(n)}$ by $p^n$? Recall that the latter denotes the partial type $\{\phi(x_1)\land\cdots\land \phi(x_n): \phi(x)\in p\}$.   In the setting of differential equations, this amounts relaxing the assumption that the generic solutions of the equations be \emph{independent}. There is reason to think better bounds might be available after dropping this assumption, but we do not know better bounds at the moment.

\section{Abelian reduction and the structure of $k$-irreducible subfields} \label{abredJM}

In this section, we explain the connections between the results of \cite{kumbhakar2025strongly} and \cite{jaoui2025abelian}, while interpreting the results of \cite{kumbhakar2025strongly} in model theoretic language. 
%By combining these approaches, we show how to strengthen results of each. 
We continue to assume that $(k,\delta)$
 is a differential field of characteristic zero.
 
The following definition has its origins in the work of Nishioka \cite{nishioka1990painleve} in his early attempts to characterize irreducibility in the sense of Painlev\'e (the most general definition was later given by Umemura \cite{Umemura1988771}).
\begin{defn} \cite[1.1]{kumbhakar2025strongly}
A finitely generated $\delta$-field extension $K$ of $k$ having $\text{tr.deg}_kK \geq 1$ is \emph{irreducible} over $k$ if for every differential field $k \subset M \subseteq K$, either $M$ is an
algebraic extension of $k$ or $K$ is an algebraic extension of $M$.  
\end{defn}

The following definition is given in \cite[Definition 2.1]{moosa2014some} in a more general model theoretic context: 
\begin{defn} \label{nofib}
A stationary type $\tp (a/A)$ \emph{admits no proper fibrations} if whenever $ c \in \dcl (Aa)\backslash \acl(A),$ then $a \in \acl (Ac).$ 
\end{defn}

In Definition \ref{nofib}, the type $p= \tp (a/A)$ is assumed to be stationary, which follows, for instance, in the case that $A$ is algebraically closed. In the context we are considering in this section, namely that $p=\tp(a/A)$ is $\mc C$-internal, working over an algebraically closed parameter set won't affect any of the main results in a significant way, since if $\bar p = \tp (a/ \acl (A)) $, then $Aut(\bar p /\mc C)$ is a finite index normal subgroup of $Aut (p/ \mc C)$ by Lemma 2.1 of \cite{jaoui2025abelian}. For the remainder of the section we will assume that $k$ is an algebraically closed differential field.

By Definition \ref{nofib}, $\tp (a/k)$ admits proper fibrations if and only if there exists $c\in k \langle a \rangle$ such that $0<\text{tr.deg}_k k\langle c\rangle <\text{tr.deg}_k k\langle a\rangle$, i.e., $k\langle a\rangle$ is not irreducible over $k$. Thus, one obtains: 
 
\begin{prop}
%Let $K = k \langle a \rangle$ be $k$-irreducible. Then $\tp (a/k)$ admits no proper fibrations. 
$K = k \langle a \rangle$ is irreducible over $k$ if and only if $\tp (a/k)$ admits no proper fibrations.
\end{prop}

One of the main results of \cite{kumbhakar2025strongly} is:

\begin{thm} \label{KS1.1} \cite[Theorem 1.1]{kumbhakar2025strongly}
Let $k \subset K \subseteq E$ be differential fields, $E$ be a Picard-Vessiot extension of $k$ and $K$ be $k$-irreducible. Then, there is a nonzero element $y \in E$ and a monic linear differential operator $L \in k[ \delta ]$ 
%of order $d$ 
such that $L(y) = 0$ and $K$ is a finite algebraic extension of $k\langle \frac{y'}{y} \rangle $.
\end{thm}

%Supposing that $E$ is of transcendence degree $1$ and the extension is not of general type, the type $p$ of a generator of $E$ is $\mc C$-internal and weakly $\mc C$-orthogonal. The binding group of $p$ is, by the Galois correspondence, a quotient of the binding group of $E$, which is linear. 

Transcendence degree one differential field extensions are always $k$-irreducible and $\mc C$-internality implies that the type of a generator of such extension is not of general type. One can thus see that Theorem \ref{KS1.1} recovers the following result of \cite[Proposition 6.5]{jaoui2025abelian}: 

\begin{prop}
    Suppose $p$ is a complete $1$-dimensional type over $k$ that is $\mc C$-internal, weakly $\mc C$-orthogonal, and $Aut(p/ \mc C)$ equal to its linear part. Then $p$ is interdefinable with the generic type of a Riccati equation. 
\end{prop}

%Theorem \ref{KS1.1} generalizes the previous result of Jaoui and Moosa to the higher order situation as long as the extension is assumed to be $k$-irreducible. 

We now recall one of the central notions from \cite{jaoui2025abelian}.

\begin{defn} \cite[Section 3]{jaoui2025abelian} By an \emph{abelian reduction} of a stationary type $p \in S(k),$ we mean a fibration $f: p \rightarrow p_{ab}$ such that \begin{enumerate}
    \item $p_{ab}$ is $\mc C$-internal and has binding group an abelian variety in $\mc C$ and 
    \item if $g: p \rightarrow q$ is a definable function with $q $ $\mc C$-internal and having binding group an abelian variety in $\mc C$, then there is a definable function $h: p_{ab} \rightarrow q$ such that $h f = g.$ We refer sometimes to $p_{ab}$ as the abelian reduction of $p$. 
\end{enumerate}
\end{defn}

In the context of this section, with $k$ algebraically closed, the binding group $G$ of 
%the extension $E$
a $\mc C$-internal type $p$ over $k$ is (isomorphic to) a connected algebraic group of $\mc C$-points. $G$ has a maximal $k$-definable linear subgroup $L$, such $L$ is normal, and $G/L$ is a connected abelian variety. $L$ is called the linear part of $G$. $G/L$ is called the abelian part of $G$. $G$ has a minimal normal definable subgroup $D$ such that $G/D$ is linear. $D$ is called the anti-linear part of $G$. The linear and anti-linear parts of $G$ commute, and $G = L \cdot D$. Jaoui and Moosa prove \cite[Theorems 3.7 and 4.2 and Corollary 4.4]{jaoui2025abelian}:

\begin{thm} \label{JM4.4} Let $p$ be a $\mc C$-internal type with binding group $G$. 
\begin{enumerate}
    \item Then $p$ has an abelian reduction $p_{ab}$. If in  addition $p$ is weakly orthogonal to $\mc C$, then the binding group of $p_{ab}$ is the abelian part of $G$ (this is one of the main results in \cite{jaoui2025abelian}, stated as Theorems 3.2 and 4.2 there).
    \item If $p$ has no proper fibrations, then we must have $G$ equal to either its linear part or its abelian part (this is a consequence of part(1) and appears in \cite[Corollary 4.4]{jaoui2025abelian}). 
\end{enumerate} 
\end{thm}

If $E$ is a strongly normal extension of $k$ and $k \subset K \subset E$ with $K=k \langle a \rangle$ being $k$-irreducible and $p= \tp (a/k)$, then it follows, by part (2) of the above theorem, that either $K$ has binding group with trivial linear part or trivial abelian part. In the case that the linear part is trivial, the binding group $A$ of $p$ is commutative. A faithful transitive action of a commutative group is regular, so $p$ is fundamental. Then Kolchin's theorem \cite[see VI.9]{KolchinDAAG} which is in our language given as Fact 6.4 of \cite{jaoui2025abelian} tells us that $p$ is interdefinable with the generic type of a full logarithmic-differential equation on $A$ over $k$. 

As the anti-linear part $D \lhd G$ is a normal subgroup, the subfield $K \subset E$ fixed by $D$ is a strongly normal extension of $k$. The Galois group of $K /k$ is $L=G/D.$ As its Galois group is linear, $K/k$ is Picard-Vessiot (cf. Fact 1.20 in \cite{meretzky2025picard}). Any subfield of $k \langle a \rangle $ of $E$ with the property the type $tp(a/k)$ has linear binding group must be a subfield of $K$. Putting together the last two paragraphs, we have recovered part (1) of the following result from \cite{kumbhakar2025strongly}. 

\begin{thm}  \label{KS1.2} \cite[Theorem 1.2]{kumbhakar2025strongly}
Let $k \subset K \subseteq E$ be differential fields with $E$ a strongly normal extension of $k$. 
%and $K$ be $k$-irreducible. 

\begin{enumerate} 

\item If $K$ is $k$-irreducible then either $K$ is an abelian extension of $k$ or $K \subset L \subset E$ where $L$ is a Picard-Vessiot extension of $k$. 
\item If the field $K$ can be embedded in a purely transcendental extension field of $k$, then $k \subseteq K \subseteq L \subseteq E$ where $L$ is a Picard-Vessiot extension of $k$. 
\end{enumerate}
\end{thm}

    We note that part (2) of Theorem~\ref{KS1.2} is well-known (and is not related to part (1)), an explicit argument can be found in \cite[Lemma 3.1]{eagles2025internality}.

On the other hand, whenever a type $p$ is $\mc C$-internal, weakly $\mc C$-orthogonal, and admits no proper fibrations, Theorem \ref{KS1.2} can be seen to imply the part (2) of Theorem \ref{JM4.4}. Indeed, when $p \in S(k)$ and $\bar k$ is the differential closure of $k,$ $p(\bar k)$ generates a strongly normal extension of $k.$ The extension is generated by some tuple of independent realizations of $p$, $\bar a = (a_1 , \ldots , a_n)$, and $\bar a \models p^{(n)}$ is fundamental in the sense of Section 6 of \cite{jaoui2025abelian}. Now for $a \models p,$ $k \langle a \rangle $ is a $k$-irreducible differential subfield of a strongly normal extension of $k$. Now part (1) of Theorem \ref{KS1.2} implies part (2) of Theorem \ref{JM4.4}. However, note that the results in \cite{kumbhakar2025strongly} do not recover the part (1) of Theorem~\ref{JM4.4} which is the stronger result.
%(which appears as Corollary 4.4 of \cite{jaoui2025abelian}). 

%The conclusions of Jaoui and Moosa with regard to this sort of inductive construction are more detailed. For instance, their Abelian reduction of the generic type of the generator of $E$ over $k$ is shown to have various properties - for instance, it has a universal property (Definition 3.4 of \cite{jaoui2025abelian}) and is well-behaved with respect to base change (Section 5 of \cite{jaoui2025abelian}). The abelian reduction $p_{ab}$ also applies to types which are not fundamental, though, one can deduce the existence of such a reduction via embedding into a fundamental type and applying the results of Kumbhakar and Srinivasan.

\appendix
\section{The inverse problem}
 \begin{center}
 {\scshape Michael F. Singer \\[2pt]}
 {\small  North Carolina State University} 
 %\textsuperscript{1}
\end{center}

\bigskip

In this appendix, we will give a proof of the following:

\begin{thm} \label{thm-A}
%\begin{quotation} 
Let $K$ be a nonconstant ordinary differential field of characteristic zero with algebraically closed constants $C$ that is  finitely generated over $C$ in the differential sense and let $G$ be a connected linear algebraic group defined over $C$. Then $G$ is the differential Galois group of a Picard-Vessiot extension of $K$.
%\end{quotation}
\end{thm}

The argument is a slight modification of  an argument due to  Kovacic as given in \cite[Proposition 3.11]{MitschiSinger1996}. %{\color{purple} (why citing Proposition 3.11 rather than Theorem 1.1 ?)} 
In that paper,  the authors show that if  a differential field $F$ containing $C$ as its field of constants and of finite, nonzero  transcendence degree over $C$, then any $G$ as above is the differential Galois group of a Picard-Vessiot extension of $F$\footnote{There is a large literature concerning inverse problems for various groups over various differential fields. See \cite{Feng_Wibmer} for a discussion of this and references.}.\\

\begin{proof}
Let $K$ be as above and denote the derivation by $D$. The differential field  $K$ is differentially algebraic over a field of the form  $C\langle y_1, \ldots y_n\rangle_D$ where $\{y_1, \ldots, y_n\}$ is a $D$-transcendence basis of this field. We can assume that $n\geq 1$ since the case of no differential transcendentals is already covered in \cite{MitschiSinger1996}.  Finally, let $E$ be the $D$-differential closure of $K$.\\

Define a new derivation $\delta = \frac{1}{Dy_1} D$ on $E$. Note that $\delta(y_1) = 1$ and so $k=C(y_1)$ is a $\delta$-differential field.  Note that as fields, $C\langle y_1, \ldots y_n\rangle_D = C(y_1)\langle Dy_1, y_2, \ldots y_n\rangle_\delta$. \\

Fix a positive integer $s > \text{tr.deg}_{C\langle y_1, \ldots y_n\rangle_D}K$. From  \cite{MitschiSinger1996}, we know that any connected linear algebraic group defined over $C$ is the differential Galois group of a PV-extension of $k$, so let $F$ be a PV extension of $k$ with $\delta$-Galois group $G^s$. Furthermore, since $k$ is a {$C_1$ field\footnote{Refer to \cite[A.52, P.372]{PuSi2003} for the definition of $C_1$ field. For example, $C(z)$ and $C((z))$ are $C_1$ field when $C$ is algebraically closed.}}, by \cite[Corollary 1.32]{PuSi2003}, we may write $F = k(g_1, \ldots , g_s)$ with $(g_1, \ldots , g_s) \in G^s(F)$ {(i.e., each $g_i\in G(F)$)}. $E$ contains the $\delta$-differential closure of $k$ so we may assume that $E$ contains $F$ since  we know that any strongly normal extension of $k$ can be embedded in its differential closure. From now on all fields consdered will be in $E$. \\

We now show that for some $i$, $K(g_i)$ is a $\delta$-PV extension of $K$ with $\delta$-differential Galois group $G$. This is done in three steps:
\begin{itemize}
\item {\bf Step 1.} {\it The fields $F$ and $k\langle Dy_1, y_2, \ldots , y_n\rangle_\delta$ are free\footnote{See \cite[Ch.\,VII,\S3]{LANG} for the definition of free.} over $k$}. A computaion shows that the  elements $Dy_1, y_2, \ldots , y_n$ are $\delta$-differentially independent over $k$. Now, assume $u _1, \ldots , u_m \in k\langle Dy_1, y_2, \ldots , y_n\rangle_\delta$ are algebraically independent over $k$ but become algebraically dependent over $F$. The polynomial relation witnessing this dependence would yield (after clearing denominators) a $\delta$-differential dependence among $Dy_1, y_2, \ldots , y_n$ \underline{over $F$}. We have that $F$ is a $\delta$-differential algebraic extension of $k$ so  \cite[Corollary 3, p.~107]{KolchinDAAG} would imply that $Dy_1, y_2, \ldots , y_n$ are $\delta$-differentially dependent over $k$ as well, a contradiction.

\item {\bf Step 2.} {\it The $\delta$-differential Galois group $H$ of $k\langle Dy_1, \ldots , y_n\rangle_\delta(g_1, \ldots , g_s)$ over \\ $k\langle Dy_1, \ldots , y_n\rangle_\delta  $ is $G^s$.} Since $H$ is a subgroup of $G^s$, all we need to show is that they have the same dimension. From \cite[Corollary 1.30(3)]{PuSi2003}, we know that $\dim_C G^s = \text{tr.deg}_kk(g_1, \ldots g_s)$. Since $F = k(g_1, \ldots ,g_s) \mbox{ and }
k\langle Dy_1, y_2, \ldots, y_n\rangle$ are free over $k$ we have that 
\begin{align*} &  \ \ \ \  \text{tr.deg}_kk(g_1, \ldots g_s \\&=\text{tr.deg}_{k\langle Dy_1, \ldots , y_n\rangle}k\langle Dy_1, \ldots , y_n\rangle(g_1, \ldots g_s) \\& = \dim_CH.
\end{align*}

\item {\bf Step 3.} {\it For some $i$, the $\delta$-Galois group of $K(g_i)$ over $K$ is $G$.}  We will argue by contradiction. Assume that for all $i$, the $\delta$-Galois group of $K(g_i)$ over $K$ is a proper subgroup of $G$. In this case, we have that $\text{tr.deg}_KK(g_i)\leq t-1$ where $t = \dim_CG$. Therefore we have   
\begin{align*}
& \ \ \ \  \text{tr.deg}_{C\langle y_1, \ldots y_n\rangle_D}K(g_1, \ldots ,g_s)\\&=\text{tr.deg}_KK(g_1, \ldots ,g_s) + \text{tr.deg}_{C\langle y_1, \ldots y_n\rangle_D}K\\&\leq
 s(t-1) + \text{tr.deg}_{C\langle y_1, \ldots y_n\rangle_D}K \\&=st - s + \text{tr.deg}_{C\langle y_1, \ldots y_n\rangle_D}K \\&< st
\end{align*}
since $s >\text{tr.deg}_{C\langle y_1, \ldots y_n\rangle_D}K$.
\sloppy As $C\langle y_1, \ldots y_n\rangle_D(g_1, \ldots , g_s) \subset K(g_1, \ldots , g_s)$, we have $\text{tr.deg}_{C\langle y_1, \ldots y_n\rangle_D}C\langle y_1, \ldots y_n\rangle_D(g_1, \ldots , g_s) <st$.  But Step 2 implies $\text{tr.deg}_{C\langle y_1, \ldots y_n\rangle_D}C\langle y_1, \ldots y_n\rangle_D(g_1, \ldots , g_s) = st$.  This gives us a contradiction. 
So there exists some $i$ such that $G$ is the $\delta$-Galois group of $K(g_i)$ over $K$.
\end{itemize}

Replacing $\delta$ by $D/Dy_1$, turns any $\delta$-PV extension of $K$ into a $D$-PV extension with the same differential Galois group.
Hence, $G$ is the $D$-Galois group of $K(g_i)_D$ over $K$.

\end{proof}

\bibliographystyle{plain}
\bibliography{research}

\end{document}